\documentclass[12pt]{article}
\usepackage{graphicx}
\usepackage{amssymb,amsmath,amsthm}
\usepackage[latin1]{inputenc}
\usepackage[english]{babel}
\newtheorem{theorem}{Theorem}[section]
\newtheorem{lemma}[theorem]{Lemma}
\newtheorem{question}[theorem]{Question}

\theoremstyle{remark}
\theoremstyle{remark}
\theoremstyle{remark}
\theoremstyle{remark}

\begin{document}

\title{$(k+1)$-sums versus $k$-sums}

\author{Simon Griffiths\, \\
  \small{IMPA, Est. Dona Castorina 110, Jardim Bot\^anico, Rio de Janeiro, Brazil}}

\maketitle

{\renewcommand{\thefootnote}{\relax} \footnotetext{
Email\,: 
    \texttt{sgriff@impa.br}}}

\begin{abstract} A $k$-sum of a set $A\subseteq \mathbb{Z}$ is an integer that may be expressed as a sum of $k$ distinct elements of $A$.  How large can the ratio of the number of $(k+1)$-sums to the number of $k$-sums be?  Writing $k\wedge A$ for the set of $k$-sums of $A$ we prove that
\[
\frac{|(k+1)\wedge A|}{|k\wedge A|}\, \le \, \frac{|A|-k}{k+1}
\]
whenever $|A|\ge (k^{2}+7k)/2$. The inequality is tight -- the above ratio being attained when $A$ is a geometric progression.  This answers a question of Ruzsa.\end{abstract}

\section{Introduction}


Given a set $A=\{a_{1},...,a_{n}\}$ of $n$ integers we denote by $k\wedge A$ the set of integers which may be represented as a sum of $k$ distinct elements of $A$.  In this paper we consider the problem of how large the ratio $|(k+1)\wedge A|/|k\wedge A|$ can be.  The upper bound 
\begin{equation}\label{eq:old}
\frac{|(k+1)\wedge A|}{|k\wedge A|}\, \le \, \frac{n}{k+1}
\end{equation}
is easily obtained using a straightforward double-counting argument.

Ruzsa~\cite{R} asked whether this inequality may be strengthened to 
\[
\frac{|(k+1)\wedge A|}{|k\wedge A|}\, \le \, \frac{n-k}{k+1}
\]
whenever $n$ is large relative to $k$.  We confirm that this is indeed the case.

\begin{theorem}\label{thm:main}  Let $A$ be a set of $n$ integers and suppose that $n\ge (k^{2}+7k)/2$.  Then 
\begin{equation}\label{eq:main}
\frac{|(k+1)\wedge A|}{|k\wedge A|}\, \le \, \frac{n-k}{k+1} \, \, .
\end{equation}
\end{theorem}

Since the ratio $(n-k)/(k+1)$ is obtained for all $k$ in the case that $A$ is a geometric progression this result is best possible for each pair $k,n$ covered by the theorem.  However, we do not believe that $n\ge (k^2 +7k)/2$ is a necessary condition for inequality \eqref{eq:main}.  Indeed, we pose the following question.

\begin{question} Does \eqref{eq:main} hold whenever $n>2k$?\end{question}

The inequality $n>2k$ is necessary.  Indeed, for any pair $k,n$ with $n/2\le k\le n-1$ the inequality \eqref{eq:main} fails for the set $A=\{1,\dots ,n\}$ (or indeed any arithmetic progression of length $n$).  To see this note that $|k\wedge A|=k(n-k)+1$ for each $k=1,\dots ,n$, and that the inequality 
\[
\frac{(k+1)(n-k)+1}{k(n-k)+1}\, \le \, \frac{n-k}{k+1}
\]
holds if and only if $k\le (n-1)/2$.  Thus, we have also verified for the case that $A$ is an arithmetic progression that \eqref{eq:main} holds whenever $n>2k$.  We also note for any set $A\subseteq \mathbb{Z}$ that \eqref{eq:main} holds trivially (and with equality) in the case that $k=(n-1)/2$.  Indeed this follows immediately from the symmetry $|k\wedge A|=|(n-k)\wedge A|\, , k=1,\dots ,n-1$.

\section{Proof of Theorem \ref{thm:main}} \label{sec:main}

The proof of Theorem \ref{thm:main} is closely related to the double-counting argument one uses to prove \eqref{eq:old}.  We recall that argument now.

Fix $k\in \{0,\dots ,n-1\}$ and a set $A=\{a_{1},\dots ,a_{n}\}$ of $n$ integers.  We say that an element $s\in k\wedge A$ \emph{extends} to $t\in (k+1)\wedge A$ if there exist distinct elements $a_1,\dots ,a_{k+1}$ of $A$ such that 
\[
s \, = \, a_1+\dots +a_k\qquad \text{and}\qquad t\, =\, a_1+\dots +a_{k+1}\, .
\]
Define the bipartite graph $G$ with vertex sets $U=\{u_s:s\in k\wedge A\}$ and $V=\{v_t:t\in (k+1)\wedge A\}$ and edge set
\[
E(G)\, =\, \{u_s v_t: s \, \text{extends to} \, t\}.
\]
We prove \eqref{eq:old} by counting $e(G)$ in two different ways:
\begin{itemize}
\item[]{(i)} $e(G)\le n|U|$, since each vertex $u_s\in U$ has at most $n$ neighbours in $V$.
\item[]{(ii)} $e(G)\ge (k+1)|V|$ since each vertex $v_t\in V$ is adjacent to each vertex $u_{t-a_i}:i=1,\dots ,k+1$, where $a_1+\dots +a_{k+1}$ is a $(k+1)$-sum to $t$.
\end{itemize}
Since $|U|=|k\wedge A|$ and $|V|=|(k+1)\wedge A|$ we obtain that
\[
(k+1)|(k+1)\wedge A|\, \le\, e(G)\, \le \, n |k\wedge A|\, ,
\]
completing the proof of \eqref{eq:old}.

The alert reader will note that the extremal cases of each of (i) and (ii) occur in rather different situations.  The inequality $e(G)\le n|U|$ may be tight only if each element $s\in k\wedge A$ extends to $s+a$ for all $a\in A$.  Equivalently, for each $s\in k\wedge A$ and $a\in A$, $s$ may be represented as a $k$-sum that does not use $a$, i.e., $s=a_1+\dots +a_k$ for distinct $a_1,\dots ,a_k\in A\setminus \{a\}$.  In particular, the inequality in $(i)$ may be tight only if each $k$-sum has at least two representations.  On the contrary, the second inequality $e(G)\ge (k+1)|V|$ may be tight only if each $t\in (k+1)\wedge A$ may be represented as a $(k+1)$-sum in a unique way.  This simple observation is the key to our proof.

We put the above observations into action by defining $Q_k \subseteq k\wedge A$ to be the set of $s\in k\wedge A$ that have a unique representation as a $k$-sum and $S=(k\wedge A) \setminus Q_k$ to be the set of $s$ with at least two representations.  We immediately obtain a new upper bound on $e(G)$, namely:
\begin{equation}\label{sthing}
e(G)\,\le \, (n-k)|U_k|+ n|S|\, =\, (n-k)|k\wedge A|+k|S|\,  .
\end{equation}
Correspondingly, one may define $Q_{k+1}$ to be the set of $t\in (k+1)\wedge A$ that are uniquely represented as a $(k+1)$-sum and $T=((k+1)\wedge A) \setminus Q_{k+1}$ to be the set of $t$ with at least two representations.  It then follows (using Lemma~\ref{lem:plus3} below) that
\begin{equation}\label{tthing}
e(G) \, \ge \, (k+1)|U_{k+1}|+(k+3)|T|\, = (k+1)|(k+1)\wedge A|+2|T|\, .
\end{equation}
Unfortunately \eqref{sthing} and \eqref{tthing} do not directly imply Theroem~\ref{thm:main} since it is non-trivial to relate $|S|$ and $|T|$.  For this reason we define a subgraph $H$ of $G$ as follows.  Recall that a pair $u_s v_t$ is an edge of $G$ if there exists a representation $s=a_1+\dots +a_k$ of $s$ as a $k$-sum of elements of $A$ and $a\in A\setminus \{a_1,\dots ,a_k\}$ such that $t=s+a$.  Include an edge $u_s v_t$ of $G$ in $H$ if and only if there exist two representations $s=a_1+\dots +a_k =b_1+\dots +b_k$ of $s$ as a $k$-sum of elements of $A$ and $a\in A\setminus (\{a_1,\dots ,a_k\}\cup \{b_1,\dots ,b_k\})$ such that $s+a=t$.  (Note: if an edge $u_s v_t$ of $G$ is included in $H$ then in particular $s\in S$ and $t\in T$.)

We begin with two lemmas.

\begin{lemma}\label{lem:hfroms} Let the set $S\subseteq k\wedge A$ and the graph $H\subseteq G$ be as defined above.  Then $e(H)\ge (n-2k)|S|$.\end{lemma}

\begin{proof} For each $s \in S$ the vertex $u_s$ has degree at least $n-2k$ in $H$.  Indeed, writing $s=a_1+\dots +a_k=b_1+\dots +b_k$ we have that $u_s v_t\in E(H)$ for each $t\in \{s+a:a\in A\setminus (\{a_1,\dots ,a_k\}\cup \{b_1,\dots ,b_k\})\}$.\end{proof}

\begin{lemma}\label{lem:plus3} Let the set $T\subseteq (k+1)\wedge A$ be as defined above.  Then $d_G(v_t)\ge k+3$ for all $t\in T$.\end{lemma}

\begin{proof} An element $t\in T$ has at least two representations $t=a_{1}+\dots +a_{k+1}=b_1+\dots +b_{k+1}$ as a $(k+1)$-sum of elements of $A$.  Furthermore the sets $\{a_1,\dots ,a_{k+1}\},\{b_1,\dots ,b_{k+1}\}$ cannot have precisely $k$ common elements (as in that case they would have different sums).  It follows that the set $B=\{a_1,\dots ,a_{k+1}\}\cup \{b_1,\dots ,b_{k+1}\}$ has cardinality at least $k+3$.  The proof is now complete since $u_s v_t$ is an edge of $G$ for each $s\in \{t-b:b\in B\}$.\end{proof}

Combining Lemma~\ref{lem:plus3} with the trivial bound $d_G(v_t)\ge d_H(v_t)$ for each $t\in T$, we deduce that
\begin{equation}\label{eq:realt}
d_G(v_t)\,\ge\, \frac{k+1}{k+3}(k+3)+\frac{2}{k+3} d_H(v_t)\, = \, k+1+\frac{2 d_H(v_t)}{k+3}\, .
\end{equation}
Consequently,
\[
e(G)\, \ge \, (k+1)|Q_{k+1}|+\sum_{t\in T} \,\, \left( (k+1)+\frac{2 d_H(v_t)}{k+3}\right) \, = \, (k+1)|(k+1)\wedge A|+\frac{2e(H)}{k+3}\,.
\]
The proof of the theorem is now nearly complete.  Indeed, applying Lemma~\ref{lem:hfroms} we obtain the bound
\[
e(G)\,\ge \,(k+1)|(k+1)\wedge A|+\frac{2(n-2k)|S|}{k+3}\, ,
\]
which combined with \eqref{sthing} yields
\[
(k+1)|(k+1)\wedge A|+\frac{2(n-2k)|S|}{k+3}\le (n-k)|k\wedge A|+k|S|\, .
\]
Now, since $n\ge (k^2+7k)/2$ the second term on the left hand side is at least the second term on the right hand side.  Thus,
\[
(k+1)|(k+1)\wedge A|\le (n-k)|k\wedge A|\, ,
\]
completing the proof of Theorem~\ref{thm:main}.

\textbf{Acknowledgements} This work began during the course of Imre Ruzsa on ``Sumsets and Structure''\cite{R2} at the CRM, Barcelona.  Thanks are due to Imre for his excellent course and stimulating problems, and to the organisers and the CRM for their hospitality.

\end{document}